\documentclass[a4paper,12pt]{article}
\usepackage{amsmath, amsthm, amstext}
\usepackage{amssymb, array, amsfonts}
\usepackage[utf8]{inputenc}
\usepackage{graphicx}

\numberwithin{equation}{section}

\def \la{\lambda}

\def \ph{\varphi}
\def \oo{\omega}

\def \D{\Delta}

\def \O{\Omega}

\def \N{\mathbb{N}}
\def \R{\mathbb{R}}

\def\n{\nabla}
\def\dd{\partial}

\def\1{1\!\!\!\!1}

\def\dom{\operatorname{Dom}}

\def\mes{\operatorname{mes}}

\theoremstyle{plain}
\newtheorem{theorem}{\bf Theorem}[section]
\newtheorem{lemma}[theorem]{\bf Lemma}

\theoremstyle{definition}

\theoremstyle{remark}
\newtheorem{rem}[theorem]{\bf Remark}

\renewcommand{\le}{\leqslant}
\renewcommand{\ge}{\geqslant}

\title{On the inequality $\mu_3 < \la_1$ for the eigenvalues of the Laplace operator}
\author{N.~D.~Filonov}
\date{}
\begin{document}
\maketitle

\begin{abstract}
In 1952, Pólya proved the inequality $\mu_2 < \la_1$ for the eigenvalues of 
Dirichlet and Neumann problems in domain $\O \subset \R^d$, $d \ge 2$.
In 2026, Tang and Zhang proved the inequality $\mu_3 < \la_1$ in 
the two-dimensional case. 
We show that $\mu_3 < \la_1$ for $d\ge 3$ as well.
\end{abstract}

\section{Formulation of the result}
Let $d \ge 2$, $\O \subset \R^d$ be a bounded domain 
such that the embedding $W_2^1 (\O) \subset L_2 (\O)$ is compact,
where $W_2^1 (\O)$ is the Sobolev space.
We consider Dirichlet and Neumann problems for the Laplace operator in $\O$:
$$
\begin{cases}
- \D \ph_k = \la_k \ph_k, \\
\left.\ph_k\right|_{\dd\O} = 0,
\end{cases}
\qquad
\begin{cases}
- \D \psi_k = \mu_k \psi_k, \\
\left.\frac{\dd\psi_k}{\dd\nu}\right|_{\dd\O} = 0.
\end{cases}
$$
Spectra of these two problems are discrete,
the eigenvalues form two sequences
\begin{equation*}
0 < \la_1 < \la_2 \le \la_3 \le \dots, \qquad \la_k \to + \infty, 
\end{equation*}
\begin{equation*}
0 = \mu_1 < \mu_2 \le \mu_3 \le \dots, \qquad \mu_k \to + \infty.
\end{equation*}
We take multiplicity into account.
In 1952, Kornhauser and Stakgold proved in \cite{KornStak} on the physical level of rigor
the inequality
\begin{equation}
\label{mu2}
\mu_2 < \la_1
\end{equation}
in two-dimensional case.
They needed it to investigate the spectrum 
of electromagnetic waves in a three-dimensional cylinder.
\footnote{For the connection between this inequality and 
the spectrum of the Maxwell operator, see, for example, 
\cite{F18, F24}.}
In his paper \cite{P52}, published in the same journal issue as \cite{KornStak},
Pólya used their ideas to provide a rigorous proof of the inequality \eqref{mu2}.
His proof works in any dimension.
This work by Pólya caused a stream of papers 
devoted to the improvements of this inequality,
see \cite{Payne55, Av86, LW86, Fr91, F04, R25, TZh}.
In particular, Levine and Weinberger proved in \cite{LW86} the inequality
\begin{equation*}
\mu_{k+d} \le \la_k \qquad \forall \ k \in \N, \ 
\text{for all convex domains} \ \O \subset \R^d.
\end{equation*}
Friedlander proved (see \cite{Fr91, F04}) the inequality
\begin{equation*}
\mu_{k+1} < \la_k
\end{equation*}
in all dimensions $d \ge 2$.
In the two-dimensional case,
Rohleder proved in \cite{R25} the inequalities
$$
\mu_{k+2} \le \la_k \qquad \text{and} \qquad \mu_3 < \la_1
$$
for simply-connected Lipschitz domains,
and Tang and Zhang proved in \cite{TZh}
the inequality $\mu_3 < \la_1$ for all planar domains with Lipschitz boundary.

We prove

\begin{theorem}
\label{t11}
Let $d \ge 3$, $\O \subset \R^d$ be a bounded domain 
such that the embedding $W_2^1 (\O) \subset L_2 (\O)$ is compact.
Then
$\mu_3 < \la_1$.
\end{theorem}

\section{Proof}
Theorem \ref{t11} follows in a rather simple way 
from the following two well-known facts.
First is that the minimal value of $\la_1$ among domains of fixed volume achieved on the ball.

\begin{theorem}[Faber-Krahn, \cite{Faber, Krahn}]
\label{t21}
Let $\O \subset \R^d$ be a bounded domain.
Then 
$$
\la_1 (\O) \ge \la_1 (B_R),
$$
where $B_R$ is a ball of the same volume, $\mes_d B_R = \mes_d \O$.
Equality is achieved only in the case of a ball.
\end{theorem}

The second fact is that the maximal value of $\mu_3$ 
among domains of fixed volume achieved on the disjoint union of two equal balls.

\begin{theorem}[Bucur-Henrot, \cite{BH}]
\label{t22}
Let $\O \subset \R^d$ be a bounded domain.
Then 
$$
\mu_3 (\O) \le \mu_2 (B_{\tilde R}),
$$
where $B_{\tilde R}$ is a ball of half volume, $\mes_d B_{\tilde R} = \frac12 \mes_d \O$.
Equality is achieved only in the case of two equal balls.
\end{theorem}

The first eigenvalue of the Dirichlet problem in a ball is well known:
\begin{equation*}
\la_1 (B_R) = \left(\frac{j_{\frac{d}2-1}}R\right)^2,
\end{equation*}
where $j_\nu$ denotes the first positive root of the Bessel function $J_\nu$. 
The equality $\mes_d B_R = \mes_d \O$ implies
$$
R = \frac{(\mes_d \O)^{1/d}}{\oo_d^{1/d}}, 
$$
where $\oo_d$ denotes the volume of the unit ball in $\R^d$.
Therefore,
\begin{equation}
\label{21}
\la_1 (\O) \ge \la_1 (B_R) = \frac{j_{\frac{d}2-1}^2 \oo_d^{2/d}}{(\mes_d \O)^{2/d}}.
\end{equation}

The second eigenvalue of the Neumann problem in a ball is also well known:
\begin{equation*}
\mu_2 (B_{\tilde R}) = \left(\frac{p_{d,1}'}{\tilde R}\right)^2,
\end{equation*}
where for natural $m$ we denote by $p_{d,m}'$ the first positive root of the function
$\left(r^{1-\frac{d}2} J_{\frac{d}2 + m -1} (r)\right)'$.
The equality $\mes_d B_{\tilde R} = \frac12 \mes_d \O$ implies
$$
\tilde R = \frac{(\mes_d \O)^{1/d}}{(2 \oo_d)^{1/d}}.
$$
Therefore,
\begin{equation}
\label{22}
\mu_3 (\O) \le \mu_2 (B_{\tilde R}) =
\frac{2^{2/d} (p_{d,1}')^2 \oo_d^{2/d}}{(\mes_d \O)^{2/d}}.
\end{equation}

The bounds \eqref{21} and \eqref{22} together yield
\begin{theorem}
\label{t23}
Let $\O \subset \R^d$ be a bounded domain 
such that the embedding $W_2^1 (\O) \subset L_2 (\O)$ is compact.
Then
$$
\frac{\mu_3 (\O)}{\la_1 (\O)} < \left(\frac{2^{1/d} p_{d,1}'}{j_{\frac{d}2-1}}\right)^2 =: c_d.
$$
\end{theorem}

The inequality is strict here, because a domain cannot be simultaneously 
a ball and the union of two balls.

Note that in the two-dimensional case this constant is $c_2 \approx 1,17$,
and this inequality is weaker than the Tang-Zhang result cited above \cite{TZh}.

Let us prove that $c_d < 1$ whenever $d \ge 3$.
If $m \ge 1$ the numbers $(p_{d,m}')^2$ coincide with the minima
of the following ratio of quadratic forms
\begin{equation}
\label{ratio}
(p_{d,m}')^2 = \min \frac{\int_0^1\left(|f'(x)|^2 x^{d-1} 
+ m (m+d-2) |f(x)|^2 x^{d-3}\right) dx}
{\int_0^1|f(x)|^2 x^{d-1} dx},
\end{equation}
defined on the set
\begin{equation*}
\dom = \left\{ f : 
\int_0^1\left(|f'(x)|^2 x^{d-1} + m (m+d-2) |f(x)|^2 x^{d-3} \right) < \infty \right\}.
\end{equation*}

\begin{lemma}
If $m\ge 1$ and $d\ge 3$ then
\begin{equation}
\label{45}
m(m+d-2) \le (p_{d,m}')^2 \le \frac{dm(m+d-2)}{d-2}.
\end{equation}
\end{lemma}

\begin{proof}
Clearly, the second term in the numerator of \eqref{ratio} satisfies
$$
\int_0^1 |f(x)|^2 x^{d-3} dx \ge \int_0^1|f(x)|^2 x^{d-1} dx ,
$$
and the first inequality \eqref{45} follows.

Next, we substitute $f(x) \equiv 1$ in \eqref{ratio}.
We get
$$
\int_0^1\left(|f'(x)|^2 x^{d-1} + m (m+d-2) |f(x)|^2 x^{d-3}\right) dx
= \frac{m(m+d-2)}{d-2},
$$
$$
\int_0^1|f(x)|^2 x^{d-1} dx = \frac1d,
$$
which implies the second inequality \eqref{45}.
\end{proof}

\begin{rem}
The lower bound \eqref{45} is just the same as (1'') in \cite{LS}.
The upper bound \eqref{45} is better than (1) in \cite{LS} for $d \ge 3$.
\end{rem}

\begin{lemma}
\label{l27}
We have 
$$
j_{\frac{d}2-1} > 2^{1/d} p_{d,1}' 
$$
whenever $d \ge 3$.
\end{lemma}

\begin{proof}
By virtue of \eqref{45} 
\begin{equation}
\label{25}
(p_{d,1}')^2 \le \frac{d(d-1)}{d-2}.
\end{equation}
For zeros of Bessel functions the bound
\begin{equation}
\label{255}
j_\nu \ge \nu + |a_1| 2^{-1/3} \nu^{1/3} ,
\end{equation}
is known, see \cite{QuWong}.
Here $a_1$ is the first zero of the Airy function $Ai$, $a_1 \approx - 2,34$.
Therefore, 
$$
j_{\frac{d}2-1} \ge \frac{d}2 - 1 + |a_1| 2^{-1/3} \left(\frac{d}2-1\right)^{1/3} 
> \frac{d}2
$$
if $d\ge 3$.
So, it suffices to prove that
\begin{equation}
\label{26}
\frac{2^{2/d} (d-1)}{d-2} < \frac{d}4.
\end{equation}
For $d \ge 7$ we have
$$
d \ge \frac{5d}{d-2} > \frac{5(d-1)}{d-2},
$$
and
$$
2^{2/d} \le 2^{2/7} < \frac54.
$$
Thus, \eqref{26} holds true if $d \ge 7$.
For small $d$ we have 
$$
\begin{array}{|c|c|c|}
\hline
d & j_{\frac{d}2-1} & p_{d,1}' \\
\hline
3 & 3,1416 & 2,0816 \\
4 & 3,8317 & 2,2999 \\
5 & 4,4934 & 2,5011 \\
6 & 5,1356 & 2,6889 \\
\hline
\end{array}
$$
We see that 
$$
j_{\frac{d}2-1} > 1,5 \,p_{d,1}' \quad \text{for} \ \ d = 3,4,5,6,
$$
and $2^{1/d} < \sqrt 2 < 1,5$.
\end{proof}

Theorem \ref{t11} follows from Theorem \ref{t23} and Lemma \ref{l27}.

\begin{rem}
From \eqref{25} and \eqref{255} we see that 
$c_d = O(\frac1d)$, $d \to \infty$.
\end{rem}


\end{document}